\documentclass[reqno]{amsart}
\usepackage{amsmath}
\usepackage{amssymb}
  \usepackage{paralist}
  \usepackage{graphics} %% add this and next lines if pictures should be in esp format
 \usepackage[colorlinks=true]{hyperref}
\usepackage{tikz}
\usetikzlibrary{calc,trees,positioning,arrows,chains,shapes.geometric,patterns,%
	decorations.pathreplacing,decorations.pathmorphing,shapes,%
	matrix,shapes.symbols}

\usepackage{mathrsfs}
\usepackage{extarrows}
\usepackage{enumerate}

\newtheorem{theorem}{Theorem}[section]
\newtheorem{corollary}[theorem]{Corollary}

\newtheorem{lemma}[theorem]{Lemma}
\newtheorem{proposition}[theorem]{Proposition}

\theoremstyle{definition}
\newtheorem{definition}[theorem]{Definition}
\newtheorem*{remark}{Remark}

\newcommand{\ep}{\varepsilon}

\newcommand{\RR}{\mathbb{R}}
\newcommand{\CC}{\mathbb{C}}
\newcommand{\NN}{\mathbb{N}}

\newcommand{\ZZ}{\mathbb{Z}}
\newcommand{\TT}{\mathbb{T}}

\newcommand{\cI}{\mathcal{I}}

\newcommand{\cR}{\mathcal{R}}
\newcommand{\cF}{\mathcal{F}}

\newcommand{\cE}{\mathcal{E}}

\newcommand{\cN}{\mathcal{N}}

\newcommand{\sC}{\mathscr{C}}

\newcommand{\sg}{\mathscr{G}}

\newcommand{\al}{\alpha}

\newcommand{\tran}{\mathrm{Tran}(X,f)}
\newcommand{\id}{\mathrm{Id}}

\tikzset{
	>=stealth',
	punktchain/.style={
		rectangle, 
		rounded corners, 
		draw=black, very thick,
		text width=9em, 
		minimum height=5em, 
		text centered, 
		on chain},
	line/.style={draw, thick, <-},
	element/.style={
		tape,
		top color=white,
		bottom color=blue!50!black!60!,
		minimum width=5em,
		draw=blue!40!black!90, very thick,
		text width=9em, 
		minimum height=5em, 
		text centered, 
		on chain},
	every join/.style={->, thick,shorten >=1pt},
	decoration={brace},
	tuborg/.style={decorate},
	tubnode/.style={midway, right=2pt},
}

\title[Zero-entropy dynamical systems with the gluing orbit property]
      {Zero-entropy dynamical systems with the gluing orbit property}

\author[Peng Sun]{}


 \email{sunpeng@cufe.edu.cn}

\begin{document}

\maketitle\

% Enter the first author's name and address:
\centerline{\scshape Peng Sun}
\medskip
{\footnotesize
% please put the address of the first author
 \centerline{China Economics and Management Academy}
   \centerline{Central University of Finance and Economics}
   \centerline{Beijing 100081, China}
} % Do not forget to end the {\footnotesize by the sign }

\bigskip

%+Abstract
\begin{abstract}

Under the assumption of the gluing orbit property, 
equivalent conditions to having zero
topological entropy are investigated.
In particular, we show that a  dynamical system has the gluing orbit property and zero topological entropy
if and only if it is minimal and equicontinuous. 
%Our proof also leads to a corollary that every
%asymptotically entropy expansive system with the gluing orbit property has dense
%intermediate metric entropies, which provides a partial positive answer to
%a conjecture of Katok.

\end{abstract}
%-Abstract

%+Contents
%\tableofcontents
%-Contents

\section{Introduction}

The specification property introduced by Bowen \cite{Bowen} plays an important
role in hyperbolic dynamics and thermodynamical formalism. 
Based on this property, 
a number of classical results
have been established, including 
growth rate of periodic orbits \cite{Bowen},
uniqueness of equilibrium states \cite{Bowen3},  
genericity of certain invariant measures \cite{Sig, Sig74}
and so on.
It is well-known
that dynamical systems with the specification property must have positive topological entropy.

Bowen's specification property is closely related to uniform hyperbolicity
and is very restrictive.
Its weak variations, called specification-like properties,
have been introduced to study broader classes of systems. 
Such properties still carry adequate information of the topological structure
of the systems and 
may survive
some non-hyperbolic behaviors.
Most results for systems with the specification property can be generalized 
under such weaker assumptions
(e.g. \cite{BTV, CT, PS}), while different phenomena
have also been observed (e.g. \cite{Pavlov, Sun19}).
More information can be found in the survey \cite{KLO}.

The gluing orbit property, a fairly new specification-like property
introduced in  \cite{BV, CT, ST}, 
has drawn much attention in recent years.
It is illustrated in \cite{BTV}  that
the gluing orbit property is compatible with zero topological
entropy, which distinguishes itself from 
stronger specification-like properties.
However, we perceive that a zero-entropy system
with the gluing orbit property should be quite simple and special.
In \cite{Sun19}, we have shown that such a  system must be minimal. 
In this article, we give a complete characterization of such systems: they are
just the minimal rotations.

\begin{theorem}\label{zeroent}
A topological dynamical system has both the gluing orbit property and zero topological entropy if and
only if it is minimal and equicontinuous, i.e. it is topologically conjugate to
a minimal rotation on a compact abelian group.
\end{theorem}

In the smooth setting, positive topological entropy implies non-zero Lyapunov
exponents. So
Theorem \ref{zeroent} indicates that the gluing orbit property is still related
to some weak hyperbolicity, except for the special case of minimal rotations.

By \cite[Theorem 1.2]{Sun19}, we already know that every minimal equicontinuous system
has the gluing orbit property.
To prove Theorem \ref{zeroent},
we need to show that a zero-entropy system with the gluing orbit property is equicontinuous. 
Here we encounter an essential difficulty caused by the variable gaps,
which is just the essential difference between the gluing orbit property and
stronger specification-like properties. Such a difficulty is not a big trouble in
 other works on the gluing orbit property (e.g. \cite{BTV, CLT, ST}).
However, it is just the variable gaps that allow weaker specification-like
properties to get along with certain non-hyperbolic behaviors.
In the sequence of works \cite{Sun19, Sunflow, Sununierg} of the author, different methods
have been applied to deal with them in studying zero-entropy systems with various weak specification-like properties.
Our proof of Theorem \ref{zeroent} is made possible by taking advantage of a sequence of notions in topological
dynamics, among which the most important one is uniform rigidity 
(see Section \ref{unirigid}).

In addition to Theorem \ref{zeroent}, we can show that a minimal system with the gluing orbit property must have zero topological entropy.
Consequently, we have the following theorem that 
summarizes our results.

\begin{theorem}\label{equiconds}
Let $(X, f)$ be a system with the gluing orbit property. 
The following are equivalent:

\begin{enumerate}
	\item $(X,f)$ has zero topological entropy. 
	\item $(X,f)$ is minimal.
	\item $(X,f)$ is equicontinuous.
	\item $(X,f)$ is uniformly rigid.
	\item $(X,f)$ is uniquely ergodic.

\end{enumerate}

\end{theorem}

\begin{figure}[h]%[5]{r}{3cm}
	\label{figcons}
	\begin{tikzpicture}
	
	\node (c1) at (0,0) {$(1)$};
	\node (c2) at (3,0) {$(2)$};
	\node (c5) at (6,0) {$(5)$};
	\node (c4) at (0,-2) {$(4)$};
	\node (c3) at (3,-2) {$(3)$};
	\draw[->, thick,] ($(c1.east)+(0,0.1)$) -- ($(c2.west)+(0,0.1)$);
	\draw[<-, thick,] ($(c1.east)+(0,-0.1)$) -- ($(c2.west)+(0,-0.1)$);
	\draw[->, thick,] (c1.south) -- (c4.north);
	\draw[->, thick,] (c4) -- (c3);
	\draw[->, thick,] (c3) -- (c1);
	\draw[->, thick,] (c5) -- (c2);
	\draw[->, thick,] (c3) -- (c5) node[tubnode] {with $(2)$};
	\end{tikzpicture}
		\caption{Conditions in Thereom \ref{equiconds}}
\end{figure}
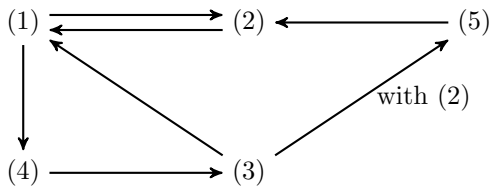

Our proof of the relations between the conditions in Theorem \ref{equiconds}
is illustrated in
Figure \ref{figcons}.
It is well-known that $(3)\implies (1)$ and $(2)(3)\implies(5)$. We have shown in
\cite{Sun19} that $(1)\implies (2)$ and $(5)\implies(2)$.
We shall prove that $(1)\implies (4)$ in Section \ref{zerour} 
 and $(4)\implies(3)$ in Section \ref{unirigeq}. 
Finally, we prove that $(2)\implies (1)$ 
in Section \ref{minizero}.

We say that a system is trivial if 
it consists of a single periodic orbit. As a corollary of Theorem \ref{equiconds}, 
every non-trivial
system with the gluing orbit property that is either non-invertible or symbolic
must have positive topological entropy. 
Furthermore, any non-trivial minimal
subshift, e.g. the strictly ergodic subshifts 
constructed in \cite{GH, HK}, 
does not have the gluing orbit property,
regardless of its entropy.

It is shown in \cite{BeTV} that $C^0$ generic systems
satisfy the gluing orbit property on isolated chain-recurrent classes.
As a complement to this result,
we are aware of one-parameter families in which generic systems are minimal
and fail to satisfy the gluing orbit property (on their 
unique chain recurrent class,
the whole space).
This is also a consequence of Theorem \ref{equiconds}.
See Section \ref{herman}.

To show that minimality implies zero entropy for systems
with the gluing orbit property (Proposition
\ref{nonmini}), we construct a compact invariant subset with small 
$\ep$-entropy.
Along with asymptotic entropy expansiveness and the
entropy denseness proved in \cite{PS}, this leads to 
a partial positive answer to the conjecture of  Katok 
that every $C^2$ diffeomorphism has ergodic measures of arbitrary intermediate  entropies.
Partial results on the conjecture have been discussed in 
\cite{GSW, KKK, QS, Sun09, Sun10, Sun12, Ures}.
See Section \ref{denseme}.

\begin{theorem}\label{denseent}
Let $(X,f)$ be an asymptotically entropy expansive system with the gluing orbit property. Then the set
$$\cE(f):=\{h_\mu(f):\text{ $\mu$ is an ergodic measure for $(X,f)$}\}$$
is dense in the interval $[0,h(f)]$.
\end{theorem}

We remark that the author has recently obtained much stronger
results than Theorem \ref{denseent}. In \cite{Sunie},
Katok's conjecture is completely proved for all systems satisfying asymptotic entropy
expansiveness and the approximate product property.
The approximate product property 
is an even weaker specification-like property
 than the gluing orbit property.
In another work \cite{Sununierg}, the author 
has also made an effort to characterize
zero-entropy systems with this property and showed that these systems must be uniquely
ergodic.
Moreover, it is shown in \cite{Sunct} 
that in Theorem \ref{denseent} the gluing
orbit property can be replaced by the decomposition introduced by Climenhaga
and Thompson \cite{CT}, which only asks the gluing orbit property 
to hold on a suitable collection of orbit segments, instead of all of them.
Such decompositions are admitted by certain DA maps \cite{CFT,
	CTmane}, for which Katok's conjecture can be verified \cite{Sunes}.
%which is admitted by certain DA maps \cite{CFT, CTmane, Sunes}.

\section{Preliminaries}

Let $(X,d)$ be a compact metric space. Throughout this article, we assume that $X$ is
infinite to avoid trivial exceptions.  Let $f:X\to X$ be a continuous map. 
Then $(X,f)$ is conventionally called a \emph{topological dynamical system} or just a \emph{system}.
We shall denote by 
$\ZZ^+$ the set of all positive integers and by $\NN$ the set of all nonnegative integers, i.e.
$\NN=\ZZ^+\cup\{0\}$. For $M\in\ZZ^+$,  we denote  by
$$\Sigma_M:=\{1,2,\cdots, M\}^{\ZZ^+}$$
the space of sequences in $\{1,2,\cdots, M\}$.

\subsection{The gluing orbit property}
\begin{definition}\label{gapshadow}
We call a sequence 
$$\sC=\{(x_j,m_j)\}
_{j\in\ZZ^+}$$
of ordered pairs in $X\times\ZZ^+$
an \emph{orbit sequence}. A \emph{gap} 
for an orbit sequence
is a sequence
$$\sg=\{t_j\}_{j\in\ZZ^+}$$
of positive integers.
For $\ep>0$, we say that $(\sC,\sg)$ can be \emph{$\ep$-traced} by $z\in
X$ if the following \emph{tracing property} (illustrated in Figure \ref{figgaps}) holds:

For every $j\in\ZZ^+$,
\begin{equation}\label{eqshadow}
d(f^{s_{j}+l}(z), f^l(x_j))\le\ep\text{ for each }l=0,1,\cdots, m_j-1,
\end{equation}
where
$$s_1:=0\text{ and }s_j:=\sum_{i=1}^{j-1}(m_i+t_i-1)\text{ for }j\ge 2.$$

\end{definition}

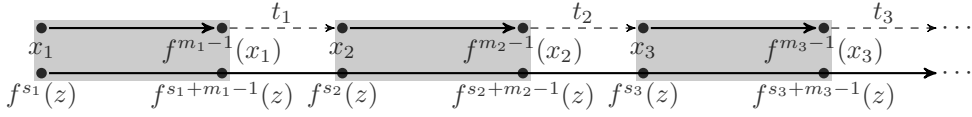
\begin{figure}[h]%[5]{r}{3cm}
	\label{figgaps}
\begin{tikzpicture}[fill opacity=0.8]

\path[fill, fill opacity=0.2,] 
(-0.1,0.1) rectangle (2.5,-0.7);
\path[fill, fill opacity=0.2,]
(3.9,0.1) rectangle (6.5,-0.7);
\path[fill, fill opacity=0.2,] 
(7.9,0.1) rectangle (10.5,-0.7);

\fill (0,0) circle (2pt);
\node at (0,-0.3) {$x_1$};
\draw [, ->, thick,] (0.1,0) -- (2.3,0);
\fill (2.4,0) circle (2pt);
\node at (2.4,-0.3) {$f^{m_1-1}(x_1)$};
\draw [, ->, dashed,] (2.5,-0) -- (3.9,-0);

\node at (3.2,0.2) {$t_1$};
\fill (4.0,0) circle (2pt);
\node at (4,-0.3) {$x_2$};
\draw [, ->, thick,] (4.1,0) -- (6.3,0);
\fill (6.4,0) circle (2pt);
\node at (6.4,-0.3) {$f^{m_2-1}(x_2)$};
\draw [, ->, dashed,] (6.5,0) -- (7.9,0);

\node at (7.2,0.2) {$t_2$};
\fill (8,0) circle (2pt);
\node at (8,-0.3) {$x_3$};
\draw [, ->, thick,] (8.1,0) -- (10.3,0);
\fill (10.4,0) circle (2pt);
\node at (10.4,-0.3) {$f^{m_3-1}(x_3)$};
\draw [, ->, dashed,] (10.5,-0) -- (11.9,-0);

\node at (11.2, 0.2) {$t_3$};
\node at (12.2,0) {$\cdots$};

\fill (0,-0.6) circle (2pt);
\node at (0,-0.9) {$f^{s_1}(z)$};
\draw [, ->, thick,] (0.1,-0.6) -- (11.9,-0.6);

\fill (2.4,-0.6) circle (2pt);
\node at (2.4,-0.9) {$f^{s_1+m_1-1}(z)$};

\fill (4.0,-0.6) circle (2pt);
\node at (4,-0.9) {$f^{s_2}(z)$};

\fill (6.4,-0.6) circle (2pt);
\node at (6.4,-0.9) {$f^{s_2+m_2-1}(z)$};

\fill (8,-0.6) circle (2pt);
\node at (8,-0.9) {$f^{s_3}(z)$};

\fill (10.4,-0.6) circle (2pt);
\node at (10.4,-0.9) {$f^{s_3+m_3-1}(z)$};

\node at (12.2,-0.6) {$\cdots$};

\end{tikzpicture}
	\caption{The tracing property}
\end{figure}

\begin{definition}\label{defgo}
	We say that $(X,f)$ 
	has \emph{the gluing orbit property},
	if for every $\ep>0$ there
	is $M=M(\ep)>0$ such that for any orbit sequence 
	$\sC$, there is a gap $\sg\in\Sigma_M$ 
	such that 
	$(\sC,\sg)$  can be $\ep$-traced.
\end{definition}

\begin{remark}
Definition \ref{gapshadow} naturally extends to finite orbit sequences.
By \cite[Lemma 2.10]{Sun19}, Definition \ref{defgo} is equivalent to the original definition
of the gluing orbit property (cf. \cite[Definition 2.8]{Sun19} or \cite[Definition 2.1]{BV}),
which requests that any finite orbit sequence can be traced.
Another difference to note is that we allow non-strict inequalities in \eqref{eqshadow}.
This provides extra convenience in Section \ref{minise} (Lemma \ref{ycompt}).

\end{remark}

For definitions and results on the other specification-like properties,
readers are referred to \cite{KLO} and \cite{Sunie}.

\subsection{Topological entropy and entropy expansiveness}
\begin{definition}
Let $K$ be a subset of $X$.
For $n\in\ZZ^+$ and $\ep>0$, a subset $E\subset K$ 
is called an \emph{$(n,\ep)$-separated set} in $K$ if
for any distinct points $x,y$ in $E$, there is $k\in\{0,\cdots,n-1\}$ such
that
$$d(f^k(x),f^k(y))>\ep.$$
Denote by $s(K,n,\ep)$ the maximal cardinality of $(n,\ep)$-separated subsets
of $K$. 
We define the 
\emph{$\ep$-entropy} as
$$h(K,f,\ep):=\limsup_{n\to\infty}\frac{\ln s(K,n,\ep)}{n}.$$
Then the \emph{topological entropy} of $f$ on $K$ is given by 
$$h(K,f):=\lim_{\ep\to0}h(K,f,\ep).$$
In particular, $h(f):=h(X,f)$ is the topological entropy of the system $(X,f)$.
\end{definition}

\begin{remark}
Note that $h(K,f,\ep)$ grows as $\ep$ tends to $0$. So we actually have
$$h(K,f)=\sup\{h(K,f,\ep):\ep>0\}.$$
\end{remark}

Readers are referred to \cite{Walters} for more discussions on topological entropy, as well as many other background materials.

\begin{definition}
For $\ep>0$ and $x\in X$, denote
$$\Gamma_\ep(x):=\{y\in X:d(f^n(x),f^n(y))<\ep\text{ for every }n\in\NN\}.$$
Let
$$h^*(f,\ep):=\sup\{h(\Gamma_\ep(x),f):x\in X\}.$$
\begin{enumerate}
\item We say that $(X,f)$ is \emph{entropy expansive} if there is $\ep_0>0$
such that $h^*(f,\ep_0)=0$.
\item 
We say that $(X,f)$ is \emph{asymptotically entropy expansive} if
$$\lim_{\ep\to 0}h^*(f,\ep)=0.$$
\end{enumerate}
\end{definition}

\begin{proposition}[{\cite[Corollary 2.4]{Bowen2}}]\label{hlocest}
For every subset $K$ and every $\ep>0$, we have
$$h(K,f)\le h(K,f,\ep)+h^*(f,\ep).$$
\end{proposition}

\subsection{Almost periodicity and Equicontinuity}

\begin{definition}
A set $A\subset\NN$ is called \emph{syndetic} if there is $L>0$ such that
$$A\cap[n, n+L-1]\ne\emptyset\text{ for every }n\in\NN.$$
\end{definition}

\begin{definition}
For $x\in X$ and $\ep>0$, denote
$$R(x,\ep):=\{n\in\ZZ^+: d(f^n(x),x)\le\ep\}.$$
A point $x\in X$ is called \emph{almost periodic} in $(X,f)$ if
$R(x,\ep)$ is syndetic for every $\ep>0$.
\end{definition}

Denote by $C(X,X)$ the space of all
continuous maps from $X$ to itself, equipped with the $C^0$-metric
$$D^0(f,g):=\sup_{x\in X}d(f(x),g(x)).$$

\begin{definition}\label{rectimes}
For $\ep>0$, denote
\begin{equation*}
R(\ep):=\bigcap_{x\in X}R(x,\ep)=\{n\in\ZZ^+: D^0(f^n, \id)
\le\ep\},
\end{equation*}
where $\id:X\to X$ is the identity map.
We say that $(X,f)$ is 
\emph{uniformly almost periodic} if
$R(\ep)$ is syndetic for every $\ep>0$.
\end{definition}

\begin{definition}
        We say that $(X,f)$ is  
\emph{equicontinuous} if 
        the family $\{f^n\}_{n=0}^\infty$ is equicontinuous, i.e.
        for every $\ep>0$, there is $\delta>0$ such that for any $x,y\in X$ 
        with $d(x,y)<\delta$, we have
        $$d(f^n(x),f^n(y))<\ep\text{ for every $n\in\NN$}.$$
\end{definition}

\begin{proposition}[{\cite[Theorem 4.38]{GH}}]\label{ecuap}
$(X,f)$ is equicontinuous if %and only if 
it is uniformly almost periodic.
\end{proposition}

\begin{remark}
If $f$ is surjective, then equicontinuity also implies uniform almost periodicity.
\end{remark}

We give a proof of Proposition \ref{ecuap} for completeness.

\begin{proof}
Assume that $f$ is uniformly almost periodic. For every $\ep>0$,
 $R(\frac\ep3)$ is syndetic.
There is $L\in\ZZ^+$ such that 
$$R(\frac\ep3)\cap[n, n+L-1]\ne\emptyset\text{ for every }n\in\NN.$$
Note that $f^k$ is uniformly continuous for each $k$.
So there is $\delta>0$ such that
$$d(f^k(x),f^k(y))<\frac\ep3\text{ whenever $d(x,y)<\delta$, for each }k=1,2,\cdots,L.$$
For any $n\in\ZZ^+$, $n>L$, 
there is $j\in R(\frac\ep3)\cap[n-L,n-1]$ and $k\in\{1,2,\cdots,
L\}$ such that $n=j+k$.
Then
$$d(f^j(x),x)\le\frac\ep3\text{ for every }x\in X.$$
So we have
\begin{align*}
d(f^n(x),f^n(y))&=d(f^{j+k}(x),f^{j+k}(y))
\\&\le d(f^j(f^k(x)),f^k(x))+d(f^j(f^k(y)),f^k(y))+d(f^k(x),f^k(y))
\\&<\ep\text{ whenever $d(x,y)<\delta$}.
\end{align*}
Hence $(X,f)$ is equicontinuous.
\end{proof}

If $(X,f)$ is equicontinuous, then the cardinality $s(X,n,\ep)$ is uniformly
bounded for all $n$. This leads to the following well-known fact.

\begin{proposition}\label{eqzeroent}
Every equicontinuous system has zero topological entropy.
\end{proposition}

\subsection{Transitivity and minimality}

\begin{definition}
Denote by
$$O(x):=\{f^n(x):n\in\NN\}$$
the (forward) orbit of $x\in X$ under $f$.
Denote by
$$\tran:=\{x\in X: \overline{O(x)}=X\}$$
the set of all points with dense orbits.
Then
\begin{enumerate}
\item We say that $(X,f)$ is \emph{topologically transitive} if $\tran$ is
dense in $X$.
This is equivalent to the standard definition which states that 
for any two nonempty open sets $U$ and $V$, there is $n\in\NN$
such that $f^n(U)\cap V\ne\emptyset$.
\item We say that $(X,f)$ is \emph{minimal} if $\tran=X$. Equivalently,
there is no nonempty proper compact invariant subset of $X$.
\end{enumerate}
\end{definition}

\begin{proposition}[{\cite{ST}}]\label{goprop}
	Let $(X,f)$ be a topological dynamical system with the gluing orbit property. Then $(X,f)$ is topologically transitive.
\end{proposition}

A system that is both equicontinuous and minimal is quite simple.
Such a system can be characterized in various ways. The following proposition
provides an incomplete summary. 
See also \cite[Chapter 3]{Aus} and \cite[Theorem 4]{HN}.

\begin{proposition}[{\cite[Theorem 3.1.7]{YHS}}]
\label{miniec}
Let $(X,f)$ be a topological dynamical system. The following are equivalent:
\begin{enumerate}
\item $(X,f)$ is equicontinuous and minimal.
\item $(X,f)$ is equicontinuous and topologically transitive.
\item $(X,f)$ is topologically conjugate to a minimal isometry.
\item $(X,f)$ is topologically conjugate to a minimal rotation
on a compact abelian group.
\item $(X,f)$ is minimal and it has (topologically) discrete spectrum, i.e.
the continuous eigenfunctions
$$\{\phi\in C(X):\phi\circ f=\lambda\phi\text{ for some }\lambda\in\CC\}$$
span $C(X)$.
\end{enumerate}
\end{proposition}

In \cite{Sun19}, we have obtained the following results.

\begin{proposition}[{\cite{Sun19}}]
Let $(X,f)$ a topological dynamical system.
\begin{enumerate}
\item If $(X,f)$ has the gluing orbit property and $h(f)=0$, then $(X,f)$ is minimal.
\item If $(X,f)$ is minimal and equicontinuous, then $(X,f)$ has the gluing orbit property.
\end{enumerate}
\end{proposition}

\subsection{Uniform Rigidity}\label{unirigid}

The notion of uniform rigidity plays a key role in 
this article.

\begin{definition}
        We say that $(X,f)$ is 
        \emph{uniformly rigid} if 
        there is a sequence $\{n_k\}_{k=1}^\infty$
        such that 
        $$\lim_{k\to\infty}D^0(f^{n_k},\id)= 0,$$ 
        i.e.
        $f^{n_k}\to \id$ uniformly.

\end{definition}

Uniform rigidity is closely related to recurrence and almost periodicity.

\begin{proposition}\label{propurs}
Let $(X,f)$ be a uniformly rigid system. Then
\begin{enumerate}
\item \label{repinf}
For every $\ep>0$,
the set $R(\ep)$ (as in Definition \ref{rectimes})
is infinite \cite[Lemma 9.2.5]{YHS}.
\item $f$ is a homeomorphism \cite[Lemma 3.2.11]{YHS}. 
\item \label{urze}
$h(f)=0$ \cite[Proposition 6.3]{GM}. 
\end{enumerate}
\end{proposition}

\section{Zero Entropy Implies Equicontinuity}
\subsection{A sufficient condition for uniform rigidity}

Our key to take advantage of uniform rigidity for systems with the gluing orbit property is to consider whether a dense orbit stays close to a shift of itself.

\begin{lemma}\label{lemmafollow}
Suppose that there are $\gamma>0$, $p\in\tran$
 and $m\in\ZZ^+$ such that
 $$d\left(f^{n}(p), f^{n}(f^m(p))\right)\le\gamma\text{ for every }n\in\NN.$$
Then we have $D^0(\id, f^m)\le\gamma$, i.e.
 $$d(x, f^m(x))\le\gamma\text{ for every }x\in X.$$
\end{lemma}

\begin{proof} Take any $x\in X$.
As $p\in\tran$, there is a sequence $\{n_k\}_{k=1}^\infty$
in $\NN$ such that
$$\lim_{k\to\infty}f^{n_k}(p)=x.$$
Then for every $n\in\NN$, we have
\begin{align*}
d(x, f^m(x))
=\lim_{k\to\infty}d(f^{n_k}(p), f^m(f^{n_k}(p))\le\gamma.
\end{align*}
\end{proof}

\begin{lemma}\label{nonrigid}
Suppose that $(X,f)$ is not uniformly rigid. Then
 there is $\gamma>0$
such that for every $p\in\tran$ and
every  $m\in\ZZ^+$, there is $\tau=\tau(p,m)\in\NN$ such that
 $$d(f^{\tau}(p), f^{\tau}(f^m(p))>\gamma.$$
\end{lemma}

\begin{proof}
Assume the result fails. Then for each $k\in\ZZ$, 
there are $p_k\in\tran$ and $m_k\in\ZZ^+$ such
that
$$d(f^{n}(p_k), f^{n}(f^{m_k}(p_k))\le2^{-k}\text{ for every }n\in\NN.$$
Then by Lemma \ref{lemmafollow}, for every $x\in X$, we have
$$d(x, f^{m_k}(x))\le2^{-k}.%\text{ for every }n\in\NN.
$$
This implies that $f^{m_k}\to \id$ uniformly as $k\to\infty$ and hence $(X,f)$ is uniformly
rigid.
\end{proof}

\subsection{Zero entropy implies uniform rigidity}\label{zerour}

Based on Lemma \ref{nonrigid}, we can show that a system with the gluing orbit property and zero topological
entropy must be uniformly rigid.

\begin{proposition}\label{unirig}
Suppose that $(X,f)$ has the gluing orbit property and it is not uniformly rigid.
Then $h(f)>0$.
\end{proposition}

The rest of this subsection is devoted to the proof of Proposition \ref{unirig}.

Suppose that $(X,f)$ has the gluing orbit property and it is not uniformly rigid.
Let $M:=M(\ep)$ be the constant in the gluing orbit property (as in Definition \ref{defgo}).
By Proposition \ref{goprop}, $(X,f)$ is topologically transitive and hence $\tran\ne\emptyset$.
By Lemma \ref{nonrigid},
there are $p\in\tran$, $\gamma>0$ 
and $0<\ep<\frac13\gamma$
such that
 for each $k=1,2,\cdots, 2M-1$, there is $\tau_k\in\NN$ such that
\begin{equation}\label{tauksep}
d(f^{\tau_k}(p), f^{\tau_k}(f^k(p))>\gamma.
\end{equation}
We fix
\begin{equation}\label{pickt}
T:=2M+\max\{\tau_k: k=1,\cdots, 2M-1\}.
\end{equation}
Note that $T$ depends exclusively on $p$ and $\ep$.
Denote $$m_1:=T+M\text{ and }m_2:=T.$$
For each $\xi=\{\xi(k)\}_{k=1}^\infty\in\Sigma_2:=\{1,2\}^{\ZZ^+}$, denote
$$\sC_\xi:=\{(p,m_{\xi(k)}+1)\}_{k=1}^\infty.$$
The gluing orbit property ensures that there are $z_\xi\in X$ and 
$$\sg_\xi=\{t_k(\xi)\}_{k=1}^\infty\in\Sigma_M$$
 such that
$(\sC_\xi,\sg_\xi)$ is $\ep$-traced by $z_\xi$.

\begin{lemma}\label{lemsep}
Let $N\ge1$.
If there is $n\in\{1,\cdots, N\}$ such that
$\xi(n)\ne \xi'(n)$, then $z_\xi$ and $z_{\xi'}$
are $((N+1)(T+2M),\ep)$-separated.
\end{lemma}

\begin{proof}
We may assume that
$\xi(n)=1$, $\xi'(n)=2$ and
$$\xi(k)=\xi'(k)\text{ for each }k=1,2,\cdots, n-1.$$

The proof divides into two cases:

\begin{enumerate}[\bf {Case }1.]
\item Assume that
$$t_k(\xi)=t_k(\xi')\text{ for each }k=1,2,\cdots,n-1.$$
In this case
we have
$$s:=\sum_{k=1}^{n-1}(m_{\xi(k)}+t_k(\xi))=\sum_{k=1}^{n-1}(m_{\xi'(k)}+t_k(\xi'))\le(n-1)(T+2M).$$
That is, the orbits of $z_\xi$ and $z_{\xi'}$ start to trace the $n$-th
orbit segments of $\sC_\xi$ and $\sC_{\xi'}$, respectively, at the same time.
Then we should note that the $n$-th orbit segments of $\sC_\xi$ and $\sC_{\xi'}$ have different lengths.
Consider the tracing properties of $z_\xi$ and $z_{\xi'}$ for the $(n+1)$-th orbit segments,
whose lengths are no less than $T+1$.
We have
\begin{align}
d(f^{s+T+M+t_n(\xi)+l}(z_\xi),f^{l}(p))&\le\ep, 
\text{ for each }l=0,1,\cdots,T;\label{eqxi}\\
d(f^{s+T+t_n(\xi')+l}(z_{\xi'}),f^{l}(p))&\le\ep,
\text{ for each }l=0,1,\cdots,T.\label{eqxip}
\end{align}
Let $r:=M+t_n(\xi)-t_n(\xi')$. 
Then $1\le r\le 2M-1$. 
By \eqref{tauksep}, we have
$$d(f^{\tau_r}(p),
f^{\tau_r}(f^{r}(p)))>\gamma.$$
By \eqref{pickt}, we have
$$r+\tau_r<2M+\tau_r\le T.$$
So we can apply \eqref{eqxi} for $l=\tau_r$ and \eqref{eqxip} for $l=r+\tau_r$.
Then
\begin{align*}
&d(f^{s+T+M+t_n(\xi)+\tau_r}(z_\xi),f^{s+T+M+t_n(\xi)+\tau_r}(z_{\xi'}))
\\ \ge{}&d(f^{\tau_r}(p), f^{r+\tau_r}(p))
\\&
-d(f^{s+T+M+t_n(\xi)+\tau_r}(z_\xi), f^{\tau_r}(p))
\\&-
d(f^{s+T+M+t_n(\xi)+\tau_r}(z_{\xi'}), f^{r+\tau_r}(p))
\\>{}&\gamma-2\ep>\ep.
\end{align*}

Moreover, we have an estimate of the separation time
\begin{align*}
s+T+M+t_n(\xi)+\tau_r&\le(n-1)(T+2M)+T+M+M+(T-2M)
\\&<(N+1)(T+2M).
\end{align*}

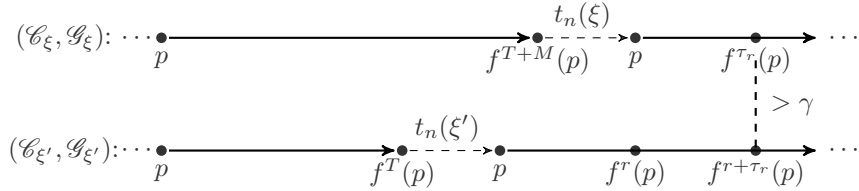
\begin{figure}[h]%[5]{r}{3cm}
	\label{figcase1}
	\begin{tikzpicture}[fill opacity=0.8]

\node at (-1,0) {$(\sC_\xi,\sg_\xi)$:};	
\node at (0,0) {$\cdots$};
\fill (0.3,0) circle (2pt);
\node at (0.3,-0.3) {$p$};
\draw [ ->, thick,] (0.4,0) -- (5.2,0);
\fill (5.3,0) circle (2pt);
\node at (5.3,-0.3) {$f^{T+M}(p)$};
\draw [ ->, dashed,] (5.4,0) -- (6.5,0);

\node at (5.9,0.3) {$t_n(\xi)$};
\fill (6.6,0) circle (2pt);
\node at (6.6,-0.3) {$p$};
\draw [ ->, thick,] (6.7,0) -- (9,0);
\fill (8.2,0) circle (2pt);
\node at (8.2,-0.3) {$f^{\tau_r}(p)$};

\node at (-1,-1.5) {$(\sC_{\xi'},\sg_{\xi'})$:};	
\node at (0,-1.5) {$\cdots$};
\fill (0.3,-1.5) circle (2pt);
\node at (0.3,-1.8) {$p$};
\draw [ ->, thick,] (0.4,-1.5) -- (3.4,-1.5);
\fill (3.5,-1.5) circle (2pt);
\node at (3.5,-1.8) {$f^{T}(p)$};
\draw [ ->, dashed,] (3.6,-1.5) -- (4.7,-1.5);

\node at (4.1,-1.2) {$t_n(\xi')$};
\fill (4.8,-1.5) circle (2pt);
\node at (4.8,-1.8) {$p$};
\draw [ ->, thick,] (4.9,-1.5) -- (9,-1.5);
\fill (6.6,-1.5) circle (2pt);
\node at (6.6,-1.8) {$f^{r}(p)$};
\fill (8.2,-1.5) circle (2pt);
\node at (8.2,-1.8) {$f^{r+\tau_r}(p)$};

\draw [dashed, thick,] (8.2, -0.3) -- (8.2,-1.5) node[tubnode] {$>\gamma$};
\node at (9.4,-1.5) {$\cdots$};
\node at (9.4,0) {$\cdots$};
		
	\end{tikzpicture}
	\caption{Separation in Case 1}
\end{figure}

\item
Assume that
$$K:=\min\{k\in\ZZ^+: t_k(\xi)\ne t_k(\xi')\}\le n-1.$$
We may assume that $t_K(\xi)>t_K(\xi')$.
Let $r:=t_K(\xi)-t_K(\xi')$.
Then 
$$1\le r\le M-1.$$
By \eqref{tauksep} and \eqref{pickt}, 
we have
$$d(f^{\tau_r}(p),
f^{\tau_r}(f^{r}(p)))>\gamma
\text{ and $r+\tau_r<T$}.$$
Let
$$s:=\sum_{k=1}^{K}(m_{\xi(k)}+t_k(\xi))=\sum_{k=1}^{K}(m_{\xi'(k)}+t_k(\xi'))
+r\le K(T+2M).$$
The tracing properties yield that
\begin{align*}
d(f^{s+l}(z_\xi), f^{l}(p))&\le\ep,
\text{ for each }l=0,1,\cdots,T;\\
d(f^{s-r+l}(z_{\xi'}),f^{l}(p))&\le\ep,
\text{ for each }l=0,1,\cdots,T.
\end{align*}
Then
\begin{align*}
d(f^{s+\tau_r}(z_\xi),f^{s+\tau_r}(z_{\xi'}))
\ge{}&d(f^{\tau_r}(p),f^{r+\tau_r}(p))
\\&
-d(f^{s+\tau_r}(z_\xi),f^{\tau_r}(p))
\\&-
d(f^{s-r+(r+\tau_r)}(z_{\xi'}),f^{r+\tau_r}(p))
\\>{}&\gamma-2\ep>\ep.
\end{align*}
Again, we have
$$s+\tau_r< K(T+2M)+ T<(N+1)(T+2M).$$

\begin{figure}[h]%[5]{r}{3cm}
	\label{figcase2}
	\begin{tikzpicture}[fill opacity=0.8]
	
	\node at (-1,0) {$(\sC_\xi,\sg_\xi)$:};	
	\node at (0,0) {$\cdots$};
	\fill (0.3,0) circle (2pt);
	\node at (0.3,-0.3) {$p$};
	\draw [ ->, thick,] (0.4,0) -- (3.4,0);
	\fill (3.5,0) circle (2pt);
	\node at (3.5,-0.3) {$f^{m_{\xi(K)}}(p)$};
	\draw [ ->, dashed,] (3.5,0) -- (5.7,0);

	\node at (4.6,0.3) {$t_K(\xi)$};
	\fill (5.8,0) circle (2pt);
	\node at (5.8,-0.3) {$p$};
	\draw [ ->, thick,] (5.9,0) -- (9,0);
	\fill (8.2,0) circle (2pt);
	\node at (8.2,-0.3) {$f^{\tau_r}(p)$};
	
	\node at (-1,-1.5) {$(\sC_{\xi'},\sg_{\xi'})$:};	
	\node at (0,-1.5) {$\cdots$};
	\fill (0.3,-1.5) circle (2pt);
	\node at (0.3,-1.8) {$p$};
	\draw [ ->, thick,] (0.4,-1.5) -- (3.4,-1.5);
	\fill (3.5,-1.5) circle (2pt);
	\node at (3.5,-1.8) {$f^{m_{\xi'(K)}}(p)$};
	\draw [ ->, dashed,] (3.6,-1.5) -- (4.7,-1.5);

	\node at (4.1,-1.2) {$t_K(\xi')$};
	\fill (4.8,-1.5) circle (2pt);
	\node at (4.8,-1.8) {$p$};
	\draw [ ->, thick,] (4.9,-1.5) -- (9,-1.5);
	\fill (5.8,-1.5) circle (2pt);
	\node at (5.8,-1.8) {$f^{r}(p)$};
	\fill (8.2,-1.5) circle (2pt);
	\node at (8.2,-1.8) {$f^{r+\tau_r}(p)$};
	
	\draw [dashed, thick,] (8.2, -0.3) -- (8.2,-1.5) node[tubnode] {$>\gamma$};
	\node at (9.4,-1.5) {$\cdots$};
	\node at (9.4,0) {$\cdots$};

	\end{tikzpicture}
	\caption{Separation in Case 2}
\end{figure}
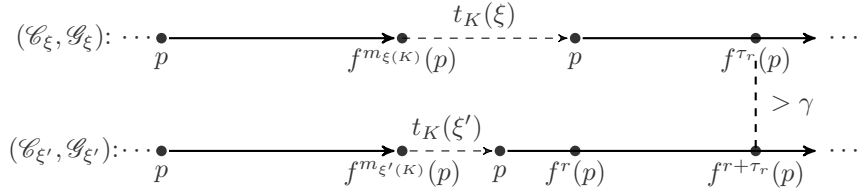

\end{enumerate}
\end{proof}

By Lemma \ref{lemsep}, for each $N$,
there is an $((N+1)(T+2M),\ep)$-separated set whose cardinality is $2^N$.
This yields that
$$h(f)\ge h(X,f,\ep)\ge\lim_{N\to\infty}\frac{\ln 2^N}{(N+1)(T+2M)}=\frac{\ln2}{T+2M}>0.$$

\subsection{Uniform rigidity implies equicontinuity}\label{unirigeq}

We remark that there are minimal and uniformly rigid systems that are weakly mixing
\cite{GM}. To complete the proof of Theorem \ref{zeroent}, we still need to
show that a uniformly rigid system
with the gluing orbit property
 is uniformly almost periodic. Then by
Proposition \ref{ecuap}, it is equicontinuous.

\begin{lemma}\label{lemsyn}
Let $(X,f)$ be a uniformly rigid  system
with the gluing orbit property. Then
for every $\ep>0$ and any $x,y\in X$, there is
$m<M:= M(\frac\ep3)$ such that
$$d(f^n(x),f^n(f^m(y)))\le\ep\text{ for every }n\in\NN.$$
\end{lemma}

\begin{proof}
By Proposition \ref{propurs}\eqref{repinf},
$R(\frac\ep3)$ is infinite.
We can find a sequence 
$\{n_k\}_{k=1}^\infty$ 
in $R(\frac\ep3)$ such that
$$n_k\to\infty\text{ and }n_k>M\text{ for every }k\in\ZZ^+.$$
For each $k$, consider the finite orbit sequence
$$\sC_k=\{(x,n_k-M+1),(y,n_k)\}.$$
There are $z_{n_k}\in X$ and $t_{n_k}\in\{1,2,\cdots, M\}$ such that
$(\sC_k,\{t_{n_k}\})$ is $\frac\ep3$-traced by $z_{n_k}$.
As the range  of $\{t_{n_k}\}_{k=1}^\infty$ is finite,
we can find a subsequence 
$\{n_l\}_{l=1}^\infty$ 
of
$\{n_k\}_{k=1}^\infty$ and $t\in\{1,2,\cdots, M\}$ such that
$$t_{n_l}=t\text{ for every }l\in\ZZ^+.$$
Then
$$M-t\in\{0,1,\cdots, M-1\}.$$ 
For every $l\in\ZZ^+$, the tracing properties yield
that
\begin{align*}
d(f^j(z_{n_l}),f^j(x))&\le\frac\ep3, \text{ for each }j=0,1,\cdots, n_l-M;\\
d(f^{n_l-M+t+j}(z_{n_l}),f^j(y))&\le\frac\ep3, \text{ for each }j=0,1,\cdots, n_l.
\end{align*}
For every $l\in\ZZ^+$, 
as $n_l\in R(\frac\ep3)$,
we have
$$d(f^n(z_{n_l}),f^{n+n_l}(z_{n_l}))\le D^0(f^{n_l},\id)\le\frac\ep3
\text{ for every }n\in\ZZ^+.$$
Hence for every $n\le n_l-M$, we have 
\begin{align*}
&d(f^n(x),f^n(f^{M-t}(y)))
\\\le{}&d(f^n(x), f^n(z_{n_l}))%+d(f^n(z_{n_l}),f^n(f^{M-t}(y)))
+d(f^n(z_{n_l}),f^{n+n_l}(z_{n_l}))+d(f^{n+n_l}(z_{n_l}),f^n(f^{M-t}(y)))
\\\le{}&\frac\ep3+\frac\ep3+\frac\ep3=\ep.
\end{align*}
This holds for every $n\in\ZZ^+$ as $n_l\to\infty$.

\begin{figure}[h]%[5]{r}{3cm}
	\label{figxy}
	\begin{tikzpicture}[fill opacity=0.8]
	
	\path[fill, fill opacity=0.2,] 
	(-0.1,0.1) rectangle (3.5,-0.7);
	\path[fill, fill opacity=0.2,]
	(5.3,0.1) rectangle (11.5,-0.7);
	\path[fill, fill opacity=0.2,] 
	(6.9,-0.5) rectangle (12.5,-1.3);
	\path[fill, fill opacity=0.2,] 
	(6.9,-1.1) rectangle (10.5,-1.9);
	
	\fill (0,0) circle (2pt);
	\node at (0,-0.3) {$x$};
	\draw [, ->, thick,] (0.1,0) -- (3.3,0);
	\fill (3.4,0) circle (2pt);
	\node at (3.4,-0.3) {$f^{n_l-M}(x)$};
	\draw [, ->, dashed,] (3.5,-0) -- (5.3,-0);

	\node at (4.4,0.2) {$t$};
	\fill (5.4,0) circle (2pt);
	\node at (5.4,-0.3) {$y$};
	\draw [, ->, thick,] (5.5,0) -- (11.3,0);
	\fill (11.4,0) circle (2pt);
	\node at (11.4,-0.3) {$f^{n_l}(y)$};

	\fill (7,-0) circle (2pt);
	\node at (7,-0.3) {$f^{M-t}(y)$};

	\fill (0,-0.6) circle (2pt);
	\node at (0,-0.9) {$z_{n_l}$};
	\draw [, ->, thick,] (0.1,-0.6) -- (12,-0.6);

	\fill (3.4,-0.6) circle (2pt);
	\node at (3.4,-0.9) {$f^{n_l-M}(z_{n_l})$};

	\fill (7,-0.6) circle (2pt);
	\node at (7,-0.9) {$f^{n_l}(z_{n_l})$};
	\node at (12.3,-0.6) {$\cdots$};
	
	\fill (7,-1.2) circle (2pt);
	\node at (7,-1.5) {$z_{n_l}$};
	\draw [, ->, thick,] (7.1,-1.2) -- (12,-1.2);
	\fill (10.4,-1.2) circle (2pt);
	\node at (10.4,-1.5) {$f^{n_l-M}(z_{n_l})$};
	\node at (12.3,-1.2) {$\cdots$};
	
	\fill (7,-1.8) circle (2pt);
	\node at (7,-2.1) {$x$};
	\draw [, ->, thick,] (7.1,-1.8) -- (10.3,-1.8);
	\fill (10.4,-1.8) circle (2pt);
	\node at (10.4,-2.1) {$f^{n_l-M}(x)$};
		
	\draw [, ->, thick, dashed, ] (3.5, -1.2) arc [radius=8, start angle=250, end angle=270];
	
	\end{tikzpicture}
	\caption{Proof of Lemma \ref{lemsyn}}
\end{figure}
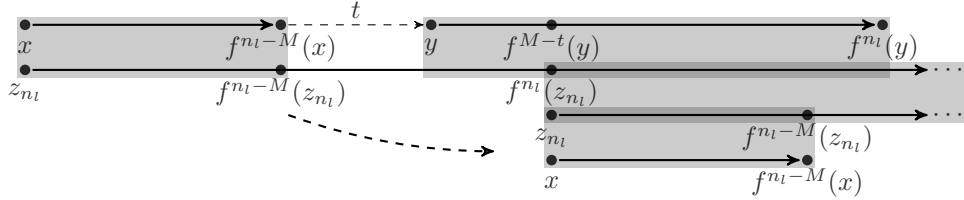
\end{proof}

\begin{proposition}\label{urgouap}
Let $(X,f)$ be a uniformly rigid system with the gluing orbit property. Then
$(X,f)$ is uniformly almost periodic (hence is equicontinuous).
\end{proposition}

\begin{proof}
By \ref{goprop}, $(X,f)$ is topologically transitive.
Take $p\in\tran$. For every $\ep>0$, let $M:=M(\frac\ep3)$.
For each $k\in\NN$, apply Lemma \ref{lemsyn} for $p$ and
$f^k(p)$. There is $m<M$ such that
$$d(p,f^m(f^k(p)))\le\ep.$$
By Lemma \ref{lemmafollow}, 
%we have
%$$d(f^n(y),f^n(f^{m+k}(y)))\le\ep\text{ for every $n\in\NN$ and for every $y\in X$}.$$
this implies that
$$D^0(\id, f^{m+k})\le\ep, \text{ i.e. }k+m\in R(\ep).$$
Hence $R(\ep)$ is syndetic for every $\ep>0$.
So $(X,f)$ is uniformly almost periodic.
\end{proof}

\section{Minimality and Small Entropy}\label{minise}
\subsection{The induced shift map}
In this subsection we introduce a symbolic system associated to a system 
with both the gluing orbit property and positive topological entropy.
Here we attempt to provide a general tool to study the uncertainty of the gaps in the gluing orbit property.
Some lemmas in this subsection may not be
indispensable in the proofs of the main results of this article.
Similar techniques have also been developed in the subsequent works \cite{Sunct, Sunie} of the author.

Suppose that $(X,f)$ has both the gluing orbit property and positive topological entropy $h(f)>0$. By Proposition \ref{goprop} and
\ref{propurs}\eqref{urze}, $(X,f)$ is topologically transitive and not uniformly rigid. 
We fix $p\in\tran$.
By Lemma \ref{nonrigid},
there is $\gamma=\gamma(f)>0$ such that %for every $p\in\tran$ and 
for each $k\in\ZZ^+$ there is
there is $\tau_k=\tau_k(p)\in\ZZ$ such that
\begin{equation}\label{nonfollow}
d(f^{\tau_k}(p), f^{\tau_k}(f^k(p))>\gamma.
 \end{equation}
Let $0<\ep<\frac13\gamma$, $M:=M(\ep)$ and
\begin{equation}
\label{taugm}
\tau>M+\max\{\tau_k: k=1,\cdots, M-1\}.
 \end{equation}
Let
$$\sC=\sC(p,\tau):=\{(p,\tau+1)\}^{\ZZ^+}.$$
For each $\sg\in\Sigma_M$, denote
$$Y_\sg:=\{x\in X:\text{
$(\sC,\sg)$ is $\ep$-traced by $x$}\}.$$
Denote
$$\Sigma=\Sigma(\tau,\ep):=\{\sg\in\Sigma_M:Y_\sg\ne\emptyset\}\subset \Sigma_M.$$
Let 
$$Y=Y(\tau,\ep):=\bigcup_{\sg\in\Sigma_M} Y_\sg=\bigcup_{\sg\in\Sigma} Y_\sg.$$
By construction, $\Sigma$ consists of all admissible gaps
with which the orbit sequence $\sC$ can be $\ep$-traced,
while $Y$ consists of all points that $\ep$-traces $\sC$ with
some admissible gaps.

\begin{lemma}\label{nepsep}
Assume that $\sg=\{t_k\}_{k=1}^\infty$ and
$\sg'=\{t_k'\}_{k=1}^\infty$ in $\Sigma_M$ such that
$(\sC,\sg)$ is $\ep$-traced by $x$, $(\sC,\sg')$ is 
$\ep$-traced by $y$,
$$t_k=t_k'\text{ for }k=1,\cdots, n-1\text{ and }t_n\ne t_n'.$$
Then $x,y$ are 
$((n+1)(\tau+M),\ep)$-separated. 
\end{lemma}

\begin{proof}
We may assume that $t_n<t_n'$.
Let $r:=t_n'-t_n$. Then
$1\le r\le M-1$.%$$
By \eqref{nonfollow}, we have
$$d(f^{\tau_r}(p),
f^{\tau_r}(f^{r}(p)))>\gamma.$$
By \eqref{taugm}, we have
$r+\tau_r<\tau$.
For
$$s:=\sum_{k=1}^{n}(\tau+t_k)\text{ and }s':=\sum_{k=1}^{n}(\tau+t_k')=s+r,$$
the tracing properties yield that
\begin{align*}
&d(f^{s'+\tau_r}(x),
f^{s'+\tau_r}(y))
\\ \ge{}&d(f^{\tau_r}(p),
f^{\tau_r}(f^{r}(p)))-d(f^{s'+\tau_r}(x),f^{\tau_r}(p))
-d(f^{s+r+\tau_r}(y),f^{r+\tau_r}(p))
\\>{}&\gamma-2\ep>\ep.
\end{align*}
This implies that $x,y$ are $((n+1)(\tau+M),\ep)$-separated, as
$$s'+\tau_r\le n(\tau+M)+(\tau-M)<(n+1)(\tau+M).$$
\end{proof}

\begin{corollary}
For every $x\in Y$ there is a unique $\sg=:G(x)$
such that $(\sC,\sg)$ is $\ep$-traced by $x$.
Hence
$\Sigma=\{G(x):x\in Y\}$.
\end{corollary}

Note that 
$\Sigma_M$ is a symbolic space on which there are a shift map
$\sigma$ and a product topology that is induced by the metric
$$\rho\left(\{t_k\}_{k=1}^\infty, \{t_k'\}_{k=1}^\infty\right)
:=2^{-\min\{k\in\ZZ^+:t_k\ne t_k'\}}.$$
We shall show that $\Sigma$ is compact and invariant under $\sigma$, hence
the map $\sigma|_\Sigma$ is a subshift. For every $x\in Y$ and $G(x)=\{t_k\}_{k=1}^\infty$,
we have that $\left(\sC, \{t_k\}_{k=1}^\infty\right)$ is $\ep$-traced by $x$.
Then $\left(\sC,\{t_{k+1}\}_{k=1}^\infty\right)$ is $\ep$-traced by $f^{\tau+t_1}(x)$,
as all orbit segments in $\sC$ are equal.
This implies that
$\sigma(G(x))=G(f^{\tau+t_1}(x))$
and $\sigma(\Sigma)\subset\Sigma$.

\begin{lemma}
The map $G:Y\to\Sigma_M$ is uniformly continuous.
\end{lemma}

\begin{proof}
Let $n\in\ZZ^+$. As $f$ is continuous and $X$ is compact, there is $\delta>0$ such that
for each $k=0,1,\cdots, (n+1)(\tau+M)$, we have
$$d(f^k(x),f^k(y))<\ep\text{ whenever $x,y\in Y$ and $d(x,y)<\delta$}.$$
So $x,y$ are not $((n+1)(\tau+M),\ep)$-separated if $d(x,y)<\delta$.
By Lemma \ref{nepsep}, this implies that 
$$\rho\left(G(x), G(y)\right)<2^{-n}\text{ whenever $x,y\in Y$ and $d(x,y)<\delta$}.$$
Hence $G$ is uniformly continuous.
\end{proof}

The following lemma shows that both $Y$ and $\Sigma$ are compact.

\begin{lemma}\label{ycompt}
Let $\{x_n\}$ be a sequence in $Y$ such that $x_n\to x$ in $X$. Then there
is $\sg\in\Sigma_M$ such that
$G(x_n)\to\sg$ and $(\sC,\sg)$ can be $\ep$-traced by $x$. 
\end{lemma}

\begin{proof}
As $\{x_n\}$ is a Cauchy sequence and $G$ is uniformly continuous, 
$\{G(x_n)\}$ is a Cauchy sequence in $\Sigma_M$. By compactness of
$\Sigma_M$, there is $\sg=\{t_k\}_{k=1}^\infty\in\Sigma_M$
such that $G(x)\to\sg$.

For each $n\in\ZZ^+$, denote 
$$G(x_n)=\{t_k(n)\}_{k=1}^\infty\text{ and }s_k(n)=\sum_{i=1}^{k-1}(\tau+t_i(n))\text{ for each
}k\in\ZZ^+.$$
For each $k\in\ZZ^+$, there is $N$ such that for every $n>N$,
$$t_j(n)=t_j\text{ for each $j=1,2,\cdots, k$ and hence }s_k(n)=\sum_{i=1}^{k-1}(\tau+t_i)=:s_k.$$
Then for each $k\in\ZZ^+$ and each $l=0,1,\cdots,\tau$, we have
\begin{align*}
d(f^{s_{k}+l}(x), f^l(p))&=\lim_{n\to\infty} d(f^{s_k+l}(x_n),f^l(p))
\\&=\lim_{n\to\infty} d(f^{s_k(n)+l}(x_n),f^l(p))\text{ (for $n>N$)}
\\&\le\ep.
\end{align*}
Hence $(\sC,\sg)$ can be $\ep$-traced by $x$.
\end{proof}

\begin{remark}
In the proof of Lemma \ref{ycompt}, we can see the advantage of the non-strict
inequality allowed in Definition \ref{defgo}, which is essential for the
compactness of $Y$. 
\end{remark}

\begin{corollary}
$Y$ is a compact subset of $X$. $\Sigma$ is a compact subset of $\Sigma_M$
and is invariant under the shift map $\sigma$.
\end{corollary}

\subsection{Invariant sets and entropy estimates}

A set of the form
$$B_n(x,\ep):=\{y\in X: d(f^k(y),f^k(x))<\ep, k=0,1,\cdots,n-1\}$$
is called an \emph{$(n,\ep)$-ball} in $(X,f)$. 
Let $E_\ep$ be a fixed $(M-1, \ep)$-separated subset
of $X$ with the maximal cardinality.
Then we must have
$$X=\bigcup_{x\in E_\ep}B_n(x,\ep),$$
i.e. $E_\ep$ is an $(n,\ep)$-spanning subset of $X$.

Denote by
$$C_{w_1\cdots w_n}=\{\{t_k\}_{k=1}^\infty\in\Sigma: t_j=w_j\text{ for each }j=1,\cdots,n\}$$
a cylinder of rank $n$ in $\Sigma$. For each cylinder $C$, denote
$$Y_C=\bigcup_{\sg\in C}Y_\sg.$$
Denote by $C(n)$ the number of different cylinders of rank $n$ in
$\Sigma$, which is equal to
the maximal cardinality of $(n,\frac13)$-separated subsets
of $\Sigma$ for the shift map $\sigma$. Then
$C(n)\le M^n$ and it is well-known that
$$h(\Sigma,\sigma)=\limsup_{n\to\infty}\frac1n\ln C(n).$$
As a corollary of Lemma \ref{nepsep}, we have
$$s(Y,(n+1)(\tau+M),\ep)\ge C(n)\text{ for every }n\in\ZZ^+.$$
Hence
$$h(Y, f)\ge h(Y,f,\ep)
\ge\limsup_{n\to\infty}\frac{\ln C(n)}{(n+1)(2\tau+M)}
=\frac{h(\Sigma,\sigma)}{2\tau+M}.$$

\begin{lemma}\label{nepestimate}
For  every $k\in\{0,\cdots, \tau+M\}$, every $n\in\ZZ^+$ and  every cylinder 
$C=C_{t_1\cdots t_{n+2}}$ in $\Sigma$ of rank $n+2$, there are at most
$|E_\ep|^{n+2}$ points in $f^k(Y_C)$ that are $(n\tau,2\ep)$-separated.
\end{lemma}

\begin{proof}
Denote
$$s_1:=0\text{ and }s_k:=\sum_{i=1}^{k-1}(\tau+t_i)\text{ for }k\ge 2.$$
Assume that $x_1,x_2\in f^k(Y_C)$ are $(n\tau,2\ep)$-separated.
There are $y_1,y_2\in Y_C$ such that
$$f^k(y_1)=x_1\text{ and }f^k(y_2)=x_2.$$
Then $y_1,y_2$ are $(n\tau+k,2\ep)$-separated and hence
$((n+2)\tau,2\ep)$-separated since 
$$k\le\tau+M<2\tau.%\text{ as }\tau>M.
$$
But in this time period $y_1$ and $y_2$
can only be separated when
their orbits are not tracing $\sC$.
So there must be
$k\le n+2$ and $1\le t\le t_k-1$ such that
$$d(f^{s_k+\tau+t}(y_1), f^{s_k+\tau+t}(y_2))>2\ep.$$
This implies that $f^{s_k+\tau+1}(y_1)$ and $f^{s_k+\tau+1}(y_2)$ are $(M-1,2\ep)$-separated.
They must belong to two $(M-1,\ep)$-balls centered at
distinct points in $E_\ep$.
Then the result follows.
\end{proof}

\begin{corollary}\label{nepesty}
For  every $k\in\{0,\cdots, \tau+M\}$ and every $n\in\ZZ^+$,
we have
$$s(f^k(Y),n\tau,2\ep)\le C(n+2)|E_\ep|^{n+2}.$$
\end{corollary}

\begin{lemma}\label{cptinv}
For $Y=Y(\tau,\ep)$, let
$$\Lambda=\Lambda(\tau,\ep):=\bigcup_{k=0}^{\tau+M-1} f^k(Y)
=\bigcup_{k=0}^{\tau+M-1}\left(\bigcup_{\sg\in\Sigma} f^k(Y_\sg)\right).$$
Then $\Lambda$ is a compact invariant set in $(X,f)$ and
\begin{equation}\label{entesti}
h(\Lambda,f,2\ep)\le\frac{\ln M+\ln |E_\ep|}{\tau}.
\end{equation}
\end{lemma}

\begin{proof}

$\Lambda$ is compact as $Y$ is compact and $f$ is continuous.

For every $x\in\Lambda$, there are $y\in Y$ and $r\in\{0,\cdots,\tau+M-1\}$ such that
$f^r(y)=x$. 
If $r<\tau+M-1$, then $f(x)=f^{r+1}(y)\in\Lambda$.
Suppose that $r=\tau+M-1$.
Assume that $\sg(y)=\{t_k\}_{k=1}^\infty$. Note that
$1\le t_1\le M$ and
$$f^{\tau+t_1}(y)\in Y_{\sigma(\sg)}\subset Y.$$
We have
$$f(x)=f^{r+1}(y)=f^{\tau+M-(\tau+t_1)}(f^{\tau+t_1}(y))\in
f^{M-t_1}(Y)\subset\Lambda.$$
So we can conclude that $f(\Lambda)\subset\Lambda$.

By Corollary \ref{nepesty}, we have for each 
$n\in\ZZ^+$, 
\begin{align*}
s(\Lambda,n\tau,2\ep)&\le
\sum_{k=0}^{\tau+M-1}s(f^k(Y),n\tau,2\ep)
\\&\le(\tau+M)C(n+2)|E_\ep|^{n+2}
\\&\le(\tau+M)M^{n+2}|E_\ep|^{n+2}.
\end{align*}
Hence
\begin{equation*}
h(\Lambda,f,2\ep)\le\limsup_{n\to\infty}\frac1{(n-1)\tau}\ln s(\Lambda,n\tau,2\ep)
=\frac{\ln M+\ln |E_\ep|}{\tau}
\end{equation*}
\end{proof}

Lemma \ref{cptinv} provides an upper bound for us to estimate the entropy
on $\Lambda$.
Note that in general, \eqref{entesti} does
not directly provide an estimate of $h(\Lambda, f)$. However,
when $\ep$ is fixed, we can fix $M=M(\ep)$ and $E_\ep$, then choose a large $\tau$
to obtain a compact invariant set $\Lambda=\Lambda(\tau,\ep)$ with small
$2\ep$-entropy.

\subsection{Minimality implies zero entropy}\label{minizero}

\begin{proposition}\label{smallepent}
Suppose that $(X,f)$ has the gluing orbit property and positive topological entropy.
Then for any $\beta,\eta>0$, there is a compact invariant set
$\Lambda$ such that $h(\Lambda,f,\eta)<\beta$.
\end{proposition}

\begin{proof}
Let $0<\eta'<\min\{\frac\eta2,\frac\gamma3\}$, where $\gamma$
is as in \eqref{nonfollow}. Let $M':=M(\eta')$. There is
$\tau'>0$ such that
$$\frac{\ln M'+\ln |E_{\eta'}|}{\tau'}<\beta.$$
Then for $\Lambda=\Lambda(\tau',\eta')$, by Lemma \ref{cptinv}, we have
$$h(\Lambda,f,\eta)\le h(\Lambda,f,2\eta')<\beta.$$ 
\end{proof}

\begin{proposition}\label{nonmini}
Suppose that $(X,f)$ has the gluing orbit property and positive topological entropy. Then
$(X,f)$ is not minimal.
\end{proposition}

\begin{proof}
Fix $\beta\in(0, h(f))$. By the definition of topological entropy, there
is $\eta>0$ such that $h(X,f,\eta)>\beta$. By
Proposition \ref{smallepent}, there is a compact invariant set $\Lambda$
such that $h(\Lambda, f, \eta)<\beta$. So $\Lambda$ is proper and hence
$(X,f)$ is not minimal.
\end{proof}

\subsection{Systems without the gluing orbit property}\label{herman}

Let $X$ be a Riemannian manifold. 
Denote by $Homeo(X)$ the space of all
homeomorphisms on $X$ equipped with the $C^0$ topology,
which is a subspace of $C^0(X,X)$.
The following is shown in \cite{BeTV}.

\begin{proposition}[{\cite[Corollary 2]{BeTV}}]\label{mostgo}
There is a  residual set $\cR$ in $Homeo(X)$ 
such that for every $f\in\cR$ and every isolated chain-recurrent class $K$ for $f$, the subsystem $(K,f|_K)$ has the gluing orbit property.
\end{proposition}

As a complement to Proposition \ref{mostgo},
we present two one-parameter families 
in which generic systems 
do not have the gluing orbit property.
These systems are minimal, hence have no isolated chain-recurrent class 
other than their whole spaces.
Both families are homeomorphic to the unit circle $\mathbb{T}:=\RR/\ZZ$
as subspaces of $Homeo(X)$.
In the first family all systems have zero topological entropy, while in the second one all systems have positive topological entropy.

For $\al\in\TT$, let 
$$f_\al(x,y):=(x+\al, y+x)\text{ for every }x,y\in\TT.$$
It is well-known that $(\TT^2, f_\al)$ has zero topological entropy while
not being equicontinuous.
Hence by Proposition \ref{zeroent}, no system in the family $\{f_\al\}_{\al\in\TT}$ has the gluing orbit property.
By \cite{Fur}, the system $(\TT^2, f_\al)$ is minimal and uniquely ergodic
if $\al$ is irrational.

In \cite{Herman}, Herman
considered a family of $C^\infty$ diffeomorphisms
$\cF=\{F_\alpha:\alpha\in\mathbb{T}\}$ on $X=\mathbb{T}\times\mathrm{SL}(2,\mathbb{R})/\Gamma$, where
$\Gamma$ is a cocompact discrete subgroup of $\mathrm{SL}(2,\mathbb{R})$. For each $\alpha\in\mathbb{T}$, $F_\al$ is a skew product such that
$$F_\alpha(x,g):=(x+\al,A_\theta(g))\text{ for every }(x,g)\in X,$$ 
where 
    $$A_\theta(g):=\left(\begin{array}{r r}
    \cos 2\pi\theta & -\sin 2\pi\theta \\ \sin 2\pi\theta & \cos
    2\pi\theta
    \end{array} \right)
    \left(\begin{array}{r r}\lambda & 0 \cr 0 &
    1/\lambda\end{array}\right)g\text{ for every }g\in\mathrm{SL}(2,\mathbb{R})/\Gamma$$
and $\lambda>1$ is a fixed real number.

\begin{proposition}[\cite{Herman}]
$h(F_\alpha)>0$ for every $\alpha\in\mathbb{T}$.
There is a dense $G_\delta$ subset $W\in\mathbb{T}$ such that
$F_\alpha$ is minimal for every $\alpha\in W$.
\end{proposition}

By Proposition \ref{nonmini}, $F_\alpha$ does not have the gluing orbit property for every $\alpha\in
W$. 
We remark that for $\al\in W$, $F_\al$ has no uniform
form of (partial) hyperbolicity but only non-zero Lyapunov exponents.
We conjecture that 
no system
in $\cF$ has the gluing orbit property.

\subsection{Denseness of intermediate metric entropies}\label{denseme}

Denote $\cI(f):=[0, h(f))$ and
$$\cE(f):=\{h_\mu(f):\text{ $\mu$ is an ergodic measure for $(X,f)$}\}.$$
The Variational Principle states that $h(f)=\sup\cE(f)$.
Katok \cite{Kat} showed that $\cE(f)\supset\cI(f)$ for any $C^{1+\al}$ diffeomorphism
on any surface and conjectured that this holds for smooth systems in all dimensions. The author
has obtained a sequence of partial results on this conjecture \cite{Sun09, Sun10, Sun12}. In \cite{QS}
Quas and Soo showed that the conjecture holds when $f$ satisfies the almost weak specification property
(also called weak specification in \cite{KLO} or other names by different
authors),
asymptotically entropy expansiveness and the small boundary property
(the latter two conditions have recently been proved unnecessary by Burguet \cite{Bur}). 
With this result, Guan, Wu and the author \cite{GSW} were able to show that Katok's conjecture holds
for certain homogeneous dynamical systems. 
The author has been making progress on the conjecture recently \cite{Sunct, Sunie, Sunes}.

By \cite{PS} (see also \cite{CLT}), every 
system with the gluing orbit property is entropy dense,
i.e. for every invariant measure $\mu$, every neighborhood $\cN$ of $\mu$ and every $\ep>0$, 
there is an ergodic measure $\nu\in\cN$
such that
$h_\nu(f)>h_\mu(f)-\ep$.
However, entropy denseness itself, even along with expansiveness, does not guarantee existence of ergodic measures
of small entropy (e.g. there are strictly ergodic subshifts of positive topological entropies \cite{HK}). To show that $\overline{\cE(f)}\supset\cI(f)$ and verify
Theorem \ref{denseent}, we need the following lemma.

\begin{lemma}\label{smallent}
Let $(X,f)$ be an asymptotically entropy expansive system with the gluing orbit property. Then for every $\beta>0$,
there is a compact invariant subset $\Lambda$ of $X$ such that
$h(\Lambda,f)<\beta$.
%Hence $\Lambda$ supports an invariant measure $\mu$ with $h_\mu(f)<\beta$.
\end{lemma}

\begin{proof}
The statement is trivial if $h(f)=0$. 

Suppose that $h(f)>0$.
By asymptotically entropy expansiveness,
there is $\eta>0$ such that
$h^*(f,\eta)<\frac\beta2$.
By Proposition \ref{smallepent}, there is a compact invariant set $\Lambda$
such that 
$h(\Lambda, f, \eta)<\frac\beta2$.
Then by Proposition \ref{hlocest}, we have
$$h(\Lambda,f)\le h(\Lambda,f,\eta)+h^*(f,\eta)
<\beta.$$
\end{proof}

\begin{proof}[Proof of Theorem \ref{denseent}]
	Let $(h, h')\subset\cI(f)$. We just need to show that there is an ergodic measure $\mu$ such that $h_\mu(f)\in (h, h')$.
	By Lemma \ref{smallent}, there is a compact invariant subset $\Lambda$ of $X$ such that $h(\Lambda,f)<\frac{h+h'}{2}$. By the Variational Principle, there are invariant measures $\nu_1$ supported on $\Lambda$ and $\nu_2$ supported on $X$ such that
	$$h_{\nu_1}(f)<\frac{h+h'}{2}\text{ and }h_{\nu_2}(f)>\frac{h+h'}{2}.$$
	By the convexity of the entropy function, there is an invariant measure $\nu_0$ that is a convex combination of $\nu_1$ and $\nu_2$ such that
	$h_{\nu_0}(f)=\frac{h+h'}{2}$.
	As $(X,f)$ is asymptotically entropy expansive, the entropy function is upper semi-continuous. So there is a neighborhood $\cN_0$ of $\nu_0$ such that
	$$h_\nu(f)<h'\text{ for every }\nu\in\cN_0.$$
	As the system is entropy dense, there is an ergodic measure $\mu\in\cN_0$
	such that $h_\mu(f)>h$, hence $h_\mu(f)\in (h, h')$.
\end{proof}

\section*{Acknowledgments}
The author is supported by National Natural Science Foundation of China (No. 11571387)
and CUFE Young Elite Teacher Project (No. QYP1902). The author would like to thank Paulo Varandas, Weisheng Wu, Dominik Kwietniak, Xueting Tian, Daniel J. Thompson and the anonymous referees.
%for helpful comments.

%+Bibliography

%-Bibliography

\end{document}